\documentclass{article}
\usepackage{amsmath,amsthm,amssymb,graphicx}
\usepackage{caption,subcaption,float}

\def\complex{{\mathbb C}}

\def\graph{{\mathcal G}}

\def\H{{\mathbb H}}

\def\vset{{\mathcal V}}
\def\eset{{\mathcal E}}

\def\dop{{\mathcal L}}

\title{A Quantum Graph FFT  \\
with applications to \\
partial differential equations on networks}
\author{Robert Carlson}

\newtheorem{thm}{Theorem}[section]
\newtheorem{cor}[thm]{Corollary}
\newtheorem{lem}[thm]{Lemma}
\newtheorem{prop}[thm]{Proposition}

\theoremstyle{definition}

\theoremstyle{remark}

\newcommand{\thmref}[1]{Theorem~\ref{#1}}

\newcommand{\lemref}[1]{Lemma~\ref{#1}}
\newcommand{\propref}[1]{Proposition~\ref{#1}}
\newcommand{\corref}[1]{Corollary~\ref{#1}}

 \numberwithin{equation}{section}

\begin{document}
\bibliographystyle{plain}
\date{2025}

\maketitle

\section{Abstract}

The  Fast Fourier Transform is extended to functions on finite graphs whose edges are identified with intervals of finite length.
Spectral and pseudospectral methods are developed to solve a wide variety of time dependent partial differential equations on domains which are 
modeled as networks of one dimensional segments joined at nodes.

\section{Introduction}

The Fast Fourier Transform (FFT) is an essential  algorithm for engineering and scientific computations.
Among the many applications are the spectral and pseudospectral methods for numerical solutions of
partial differential equations \cite{Forn,Hesthaven}.  When FFT-based spectral methods are directly applicable, they can provide remarkable
efficiency and accuracy.  While less remarkable,  FFT-based pseudospectral methods, which apply in broader contexts, can still offer important advantages, particularly
when relatively high frequency components are important.  This work introduces a quantum graph FFT algorithm (QGFFT),  
an important extension of the FFT with broad application to time dependent partial differential equations on networks. 

A great variety of natural and manufactured structures can be effectively modeled as networks of one dimensional segments joined at nodes. 
Examples range from neural and vascular networks,  river systems, and spider webs, to traffic networks, lattice and honeycomb structures, and carbon nanotubes \cite{Kuchment}.
In such settings partial differential equations (PDEs) on networks, including the heat, wave, and Schr\"odinger equations, and their linear and nonlinear 
variants, offer a powerful approach for the study of time dependent evolutions.  
  
The QGFFT is based on the spectral theory of a second derivative operator, $\dop f = -f'' $.
It has long been recognized \cite{ BK, Cattaneo, vonBelow} that the spectral theory of $\dop $ is dramatically simplified  
when a finite metric graph $\graph $ has all edges of equal length, taken here to be $1$, with standard Kirchhoff conditions at the vertices.
On each edge the eigenfunctions are elementary, while the sequence of eigenvalues has a simple arithmetic structure determined by a finite set of combinatorial data.  
Graphs $\graph$ with integer edge lengths can be reduced to the equilateral case by inserting vertices of degree two.
The QGFFT blends this equilateral graph spectral theory with the conventional FFT algorithm.

The QGFFT developed in this work offers significant advantages for computing solutions of many network PDEs.
For networks with integer edge lengths, constant coefficient heat, wave, and Schr\"odinger equations
can be solved with extraordinary accuracy and efficiency via the spectral method.  
Many other equations, including the Schr\"odinger equation with nonconstant potential, can be treated with a (Strang) splitting method.   
For networks with noninteger edge lengths, an elementary change of variables converts the original network to an equilateral one.
Time dependent PDEs, including those with nonconstant coefficients, can then be solved using a pseudospectral method, which
is particularly effective when accurate treatment of high frequency components is important.

Recent decades have seen the rapid and extensive development of quantum graphs \cite{BK,Kurasov}, motivating 
other studies of numerical techniques for solving PDEs on networks.
A linear finite element method focusing on metric graphs with Kirchhoff conditions at the vertices is introduced in \cite{Arioli}. 
 A library of functions using finite differences to approximate the Laplace operator is described in \cite{Besse}; the main demonstrations
are ground state calculations and dynamics of nonlinear Schr\"odinger equations.  Another library for quantum graph computations
is described in \cite{Goodman}; analysis of nonlinear Schr\"odinger equations is again an important motivation.
A spectral method based on highly accurate calculation of eigenvalues and eigenfunctions for the quantum graph Laplace operator, but without the
advantages of the FFT, is reported in \cite{Brio}. 

The QGFFT algorithm and illustrative applications are described in the subsequent sections.
The second section introduces equilateral quantum graphs and the regularity of their spectral theory.
The QGFFT algorithm is described.  For the benefit of readers who are not familiar with equilateral quantum graphs, 
some of the essential, but well known material, is reviewed.  This is followed by new results which link orthonormal eigenvectors 
from a discrete graph matrix with orthonormal quantum graph eigenfunctions.  These new results simplify the algorithm 
outlined in \cite{CarlsonFFT} considerably.

The final section describes QGFFT computations using Python software.  
Exactly solvable problems are used to test the accuracy of the spectral and pseudospectral methods,
to compare them with difference methods, and to consider the sensitivity of the pseudospectral method to 
edge lengths.  Various techniques are applied to the time dependent Schr\"odinger equation on networks.
The spectral method is used to treat the interaction of a wave packet with a geometric bound state.  
The Strang splitting method is used for the linear Schr\"odinger equation with added potential.
Finally, the pseudospectral method is used to treat the interaction of a wave packet with a geometric obstacle with unequal edge lengths.
Additional applications of the spectral method to linear and nonlinear variants of the heat and wave equation as well as the Schr\"odinger equation are treated in
an earlier manuscript \cite{CarlsonArXiv}

It is a pleasure to acknowledge helpful conversations with Denis Silantyev.

\section{Equilateral quantum graphs}

\subsection{Eigenfunction expansion}

Let $\graph $ denote a finite graph with vertex set $\vset $ having $N_{\vset}$ vertices, and edgeset $\eset $ with $N_{\eset}$ edges.
$\graph $ is assumed to be connected and simple (no loops or multiple edges joining a pair of vertices).
All vertices have degree at least two.  (In some cases \cite{CarlsonArXiv} this restriction can be circumvented.)
In the usual manner of equilateral  quantum graphs, the edges of $\graph $ are identified with the 
interval $[0,1]$.  One may then define the Hilbert space $L^2(\graph ) = \oplus_{e \in \eset}L^2(e) $   
with inner product 
\[ \int_{\graph } f(x) \overline{g(x)} \ dx = \sum _{e \in \eset} \int_e  f_e(x) \overline{g_e(x)} \ dx ,\]
where each component $f_e:[0,1] \to \complex $ is a function in the usual Lebesgue Hilbert space $L^2[0,1]$ of square integrable functions.  

On a domain of sufficiently differentiable functions there is a self-adjoint differential operator $\dop $ which acts by $\dop f = - f''$.
The domain of $\dop $ is defined using a standard set of (Kirchhoff) vertex conditions which require $f$ to be continuous on $\graph $.
In addition, if $v \in \vset$ and local coordinates for the edges incident on $v$ identify $v$ with $0$, functions in the domain of $\dop $
must satisfy
\begin{equation} \label{derive}
\sum _{e \sim v} f'(v) = 0.
\end{equation}   
These vertex conditions are common for the heat, wave and Schr\"odinger equations.
Details about the precise domain and self-adjointness of $\dop $ are presented in \cite{BK} and \cite{Kurasov}.
The main consequence is that $\dop $ has a discrete spectrum consisting of nonnegative eigenvalues
$0 = \lambda_0 < \lambda _1 \le \lambda _2 \le \dots $, $\lambda _{\eta} \to \infty $, with a complete orthonormal basis
of eigenfunctions.

The eigenvalues and eigenfunctions of $\dop $ on an equilateral graph can be presented in a particularly effective manner.  The main ideas are sketched here,
with more details provided below.  To begin, let $0 = \lambda_0 < \lambda _1  \le \dots \le \lambda _K \le (2\pi)^2 $ denote the eigenvalues of $\dop $, listed with multiplicity, which are less than or equal to $(2\pi )^2$.
For $k = 0,\dots , K$ define {\it fundamental frequencies} $\omega _{k,0} = \sqrt{\lambda _k} \ge 0$.
The eigenvalue $0$, with constant eigenfunctions, is a special case.  For $k = 1,\dots , K$ and positive integers $m$, the {\it higher frequencies} are $\omega _{k,m} = \omega _{k,0} + 2m\pi $.
The eigenvalues of $\dop $ are precisely $\{ 0, \omega _{k,m}^2  \ | \  k = 1,\dots , K; m = 0,1,2,\dots \}$.

An orthonormal set of eigenfunctions $\{ \Psi _{k,0} \ | \  k = 0,\dots , K \}$ with eigenvalues $\lambda _k$ is constructed first.
On an (oriented) edge $e$ of $\graph $ the eigenfunction $\Psi _{k,0}$ has the form
\begin{equation} \label{fundefunc}
\Psi _{k,0}(x) = \gamma _{k,e}\exp(i\omega _{k,0} x) + \delta _{k,e} \exp(-i \omega _{k,0} x) , \quad 0 \le x\ \le 1. 
\end{equation}
For  $k = 1,\dots , K$ and $m = 0,1,2, \dots $ the functions $\Psi _{k,m}$ are then defined to have values on $e$ given by
\[\Psi _{k,m}(x) = \gamma _{k,e}\exp(i\omega _{k,m} x) + \delta _{k,e} \exp(-i \omega _{k,m} x) ,  \quad \omega _{k,m} = \omega _{k,0} + 2\pi m . \]
The functions $\Psi _{k,m}$ are eigenfunctions of $\dop$ which comprise a complete orthonormal basis for $L^2(\graph )$.
The Fourier coefficients of $f \in L^2(\graph )$ are
\begin{equation} \label{alphadef}
\alpha _{k,m} = \int_{\graph }f\overline{\Psi _{k,m}} 
\end{equation}
\[ = \sum_{e \in \eset} \int_0^1 f_e(x)\bigl [ \overline{\gamma _{k,e}\exp(i\omega _{k,m} x) + \delta _{k,e} \exp(-i \omega _{k,m} x)}\bigr ] \  dx .\]

For numerical calculations the eigenfunction expansion is truncated and Fourier coefficients $\alpha _{k,m} $ are approximated using uniform sampling and the trapezoidal rule.  
To take advantage of the traditional FFT algorithm \cite[p. 379-399]{Briggs} for computation of the discrete Fourier transform, assume $N = 2^J$ for a nonnegative integer $J$. 
For each edge the $N+1$ sample points are $x_n = n/N$ for $n = 0,\dots ,N$.  Discrete Fourier coefficients are defined as
\begin{equation} \label{betadef}
\beta _{k,m} =  \frac{1}{2N} \sum_{e \in \eset} [ f_e(0)\overline{ \Psi_{k,m} }(0) + f_e(1)\overline{ \Psi_{k,m} }(1) + 2\sum_{n=1}^{N-1} f_e(\frac{n}{N})\overline{\Psi_{k,m}(\frac{n}{N})}] 
\end{equation}
\[ =  \frac{1}{2N}\sum_{e \in \eset} \overline{\gamma _{k,e}} \bigl [ f_e(0) + f_e(1)e^{-i\omega _{k,0}} + 2\sum_{n=1}^{N-1} f_e(\frac{n}{N}) e^{-i\omega _{k,0}n/N} e^{-2\pi i mn/N} \bigr ] \]
\[ +  \frac{1}{2N}\sum_{e \in \eset} \overline{\delta _{k,e}} \bigl [ f_e(0) + f_e(1)e^{i\omega _{k,0}} + 2\sum_{n=1}^{N-1} f_e(\frac{n}{N})e^{i\omega _{k,0}n/N}e^{2\pi imn/N} \bigr ] \]
 In fact, only the values $m = 0,\dots ,N/2 -1$ are needed.
Fixing $k$, the coefficient sequence $\{ \beta _{k,m} |  m = 0,\dots , N/2 - 1 \}$ 
can be evaluated using the terms
\[ f_e(0) + f_e(1)e^{-i\omega _j} +\sum_{n=1}^{N-1} f_e(\frac{n}{N}) e^{-i\omega _jn/N} e^{-2\pi i mn/N},\] 
and 
\[f_e(0) + f_e(1)e^{i\omega _j} + \sum_{n=1}^{N-1} f_e(\frac{n}{N})e^{i\omega _{k,0}n/N}e^{2\pi imn/N}. \]
The first of these is a  discrete Fourier transform of the sequence 
\[f_e(0) + f_e(1)e^{-i\omega _j} , f_e(\frac{1}{N})e^{-i\omega _{k,0}1/N}, \dots , f_e(\frac{N-1}{N} )e^{-i\omega _{k,0}(N-1)/N}  ,\]
and similarly for the second.

The inversion process recovers the sample values $f_e(n/N)$ from the coefficients $\beta _{k,m}$.
There are exceptional terms for frequencies $\omega = 0,N/2$.  On an edge $e$,
\begin{equation} \label{IFFT}
f_e(n/N) = \sum_{k,m} \beta _{k,m}\Psi _{k,m}(n/N)  
\end{equation}
\[ = c_0  + c_{N/2}\cos(n\pi ) + \sum_{k=1}^K e^{i\omega _{k,0} n/N} \gamma _{k,e} \sum_{m = 0}^{N/2 -1} \beta _{k,m}e^{2\pi imn/N} \]
\[+ \sum_{k=1}^K e^{-i\omega _{k,0}n/N} \delta _{k,e} \sum_{m = 0}^{N/2 -1} \beta _{k,m} e^{ -2\pi imn/N}  . \] 
The traditional FFT can be used again.  Notice here that with $k$ fixed the $N/2$ terms  $\beta _{k,m} $ are used to produce 
$N+1$ sample values $n/N$ on $e$.

\subsection{The quantum graph FFT algorithm}

The QGFFT is an accurate and efficient method of computing the coefficients $\alpha _{k,m}$
for a suitable partial sum of the expansion.  An inverse transform algorithm is also given.
The QGFFT and its inverse take advantage of the traditional (Python) FFT algorithm. 
With each edge identified with $[0,1]$, input functions are defined by their values at the
$N+1 = 2^J+1$ uniformly spaced sample points, $x_m = m/N$, $m=0,\dots ,N$.
The algorithm uses several standard computations from linear algebra; Python versions were used.   
The main steps of the algorithm are as follows.

1. As preliminary steps, the input graph $\graph $ is represented by a (weighted) adjacency matrix $A_0$ with nonnegative integer entries representing the edge lengths.
Edges with length greater than one are replaced by paths with edges of length $1$.  
Each (directed) edge is then represented by a pair of adjacent vertices (smaller vertex index first).

2. As a first linear algebra step, begin by constructing the equilateral graph adjacency matrix $A_1$ and diagonal degree matrix $T$.  
The orthonormal eigenvectors  and eigenvalues $\nu $, listed with multiplicity, of the real symmetric matrix $I - T^{-1/2}A_1T^{-1/2}$ are computed.
Multiplication by $T^{-1/2}$ converts the initial eigenvectors to eigenvectors of $I - T^{-1}A_1$.
Except for the values $\nu = 0,2$, eigenvalues $\nu$ will be subsequently converted to fundamental frequencies $\omega $ and eigenvalues $\lambda = \omega ^2$ of the quantum graph.  

3. For the second linear algebra  step the quantum graph eigenvalues $n^2\pi ^2$ with $n =0,1,2$  get a separate treatment.
The frequency $0$ is always present, with constant eigenfunctions.  
An eigenfunction with frequency $\pi$ has the form $A_e\cos(\pi x) + B_e\sin (\pi x)$ on each directed edge $e$.
Global eigenfunctions must satisfy the continuity conditions and the derivative condition at each vertex.  Using a matrix $M$ encoding the vertex conditions,
eigenfunctions with frequency $\pi$ can be identified with solutions of an equation $MX = 0$.  The  frequency $2\pi $ case is handled similarly.      
 
4. Except for the values $\nu = 0,2$, eigenvalues $\nu$ are converted to fundamental frequencies $\omega $ of the quantum graph.  
This is a one to two conversion; fundamental frequencies $\omega _1,\omega _2$ are given by $\omega _1 = \cos ^{-1}(1 - \nu )$ with $0 < \omega _1 < \pi$, and $\omega _2 = 2\pi - \omega _1$. 
For each edge and each fundamental frequency, quantum graph eigenfunction coefficients are computed.
Using the data from step 2, let $a$ be the value of a $I - T^{-1}A_1$ eigenvector at vertex $v$ and $b$ the value at the adjacent vertex $w$.
For $\nu \not= 0,2$, or $\omega \not= n\pi$, a quantum graph eigenfunction has the form $c_1\cos (\omega x) + c_2\sin(\omega x)$ with
$c_1 = a$, $c_2 = (b - a\cos(\omega ))/\sin (\omega ) $.
The coefficients $c_1,c_2$ are converted to coefficients for the exponential functions $\exp(\pm i \omega x)$.

For each edge and quantum graph eigenvalues $\pi ^2$, $4\pi ^2$,  listed with multiplicity, the quantum graph eigenfunctions found in step 3 solutions are also represented 
by coefficients of $\exp(\pm i \omega x)$.  The eigenvalue $\lambda = 0$ is treated as a special case.

5. Two  traditional (Python) FFTs are used to compute the coefficients $\beta _{k,m}$ as indicated in \eqref{betadef}.
The output is an array of $N/2$ coefficients for each fundamental frequency $\omega _{k,0}$.

6. To compute the inverse QGFFT, the coefficients for each fundamental frequency are extended by zeros to length $N$ and a traditional FFT 
is performed.  As indicated in \eqref{IFFT} this data is then summed over the fundamental frequencies to output the spatial samples.

\section{Supporting analysis}

Equilateral quantum graphs were studied in \cite{Cattaneo, vonBelow}, with additional material in \cite{BK, CarlsonFFT, Kurasov}.  
Relevant portions are presented below to explain some of the claims above.  
When a QGFFT was discussed in \cite{CarlsonFFT}, the author believed that the eigenfunctions $\Psi _{k,m}$ might not be orthonormal in all cases.  
The justifications in \thmref{eqdot}, \thmref{inprod2}, and \corref{freqinc} below seem to be new.
   
\subsection{Eigenvalues and eigenvectors}

Let $E(\lambda )$ denote the eigenspace for $\dop $ with eigenvalue $\lambda $, with $y(x) \in E(\lambda )$.
For $\omega = \sqrt{\lambda } > 0 $ and nonnegative integers $m$, define the linear mapping $S_m:E(\omega ^2) \to E([\omega +2m\pi ]^2)$ by the frequency shift
$S_my(x,\omega ^2) = y(x,[\omega +2m \pi]^2)$.  On each edge $-D^2 y(x,[\omega + 2m\pi ]^2) = [\omega + 2m\pi ]^2 y(x,[\omega + 2m\pi ]^2)$.
It is easy to check that $y(x,[\omega  + 2m\pi ]^2) $ is continuous on $\graph $ and satisfies \eqref{derive}.
Consequently, except for $E(0)$, all $\dop $ eigenspaces $E(\lambda )$ are generated by frequency shifts of the eigenspaces with eigenvalues $0 < \lambda  \le (2\pi )^2$,
as the next result shows. 

\begin{prop} \label{rooteval}
Suppose $\graph $ is a finite simple connected equilateral graph.

The eigenspace $E(0)$ is spanned by the constants. 
Suppose $\omega ^2$ is an eigenvalue of $\dop $, with $0 < \omega \le 2\pi $.  
If $m$ is a nonnegative integer, then $S_m:E(\omega ^2) \to E([\omega +2m\pi ]^2)$ is an isomorphism of eigenspaces.
\end{prop}

\begin{proof}
For $E(0)$ one notes that for $f$ in the domain of $\dop$,
\[\int_{\graph } (\dop f)\overline{f}  = \int_{\graph} |f'|^2,\]  
which is strictly positive unless $f$ is constant.

Now consider $0 < \omega \le 2\pi $. 
Since $\omega $ and $\omega +2m\pi$ are both nonzero, the functions $\cos(\omega x) ,\sin (\omega x)$ and $\cos([\omega +2m\pi ]x) ,\sin ([\omega +2m\pi ]x)$ 
are independent on each edge.  Thus $S_m: E(\omega ^2) \to E([\omega +2m\pi ]^2)$ and $S_m^{-1}:E([\omega +2m\pi ]^2)  \to E(\omega ^2)$
both have kernel $0$, establishing invertibility. 
\end{proof}

With the exception of the values $\omega _{k,0} \in \{ 0,\pi ,2\pi \}$, determination of the fundamental frequencies $\omega _{k,0}$ is a straightforward 
problem of linear algebra.  Functions $f:\vset \to \complex$ defined on the vertex set of $\graph $ form a Hilbert space $\H $ with the usual arithmetic and the degree weighted inner product  
\begin{equation} \label{inprod}
\langle f,g \rangle  = \frac{1}{2}\sum_{v} deg (v) f(v) \overline{g(v)}.
\end{equation}
Given a vertex $v \in \graph  $, let $w_1,\dots ,w_{deg (v)}$ be the vertices adjacent to $v$.
The adjacency operator is  
$Af(v) = \sum_{i=1}^{deg(v)} f(w_i)$,
the identity is $I$, and the degree operator is 
$Tf(v) = deg(v) f(v)$.

The discrete Laplace operator is defined as 
\begin{equation} \label{Deltadef}
\Delta  = I - T^{-1}A.
\end{equation}
The adjacency operator is self-adjoint with respect to the usual dot product, so 
\[\langle T^{-1}A f,g\rangle  = \frac{1}{2}\sum_{v} deg (v) (T^{-1}Af)(v) \overline{g(v)}
= \frac{1}{2} \sum_{v}  (Af(v)) \overline{g(v)},\]
implying that $\Delta $ is self-adjoint on $\H $.  Since  
\[ \Delta  = (I - T^{-1}A) = T^{-1/2}( I - T^{-1/2} AT^{-1/2})T^{1/2},\]
$\Delta $ is similar, in the sense of matrix conjugation, to the Laplace operator  $I - T^{-1/2} AT^{-1/2}$
treated in \cite[pp. 3-7]{Chung}, which is self adjoint with respect to the dot product.
In particular the eigenvalues of $\Delta $ are real and nonnegative, with $0$ having an eigenspace
spanned by the constants.  

\begin{prop} \label{combspec}
Suppose $\graph $ is a finite simple equilateral graph with edge lengths $1$.
If $\lambda  \notin \{ n^2 \pi ^2  \} $,
then $\lambda $ is an eigenvalue of $\dop $
if and only if $\nu = 1 - \cos (\sqrt{\lambda }) $
is an eigenvalue of $\Delta $.  
Evaluation at the vertices of $\graph $ gives a linear bijection between the $\lambda $
eigenspace of $\dop $ and the $\nu $ eigenspace of $\Delta $.
\end{prop}  

\begin{proof}
Suppose $y(x)$ is an eigenfunction of $\dop $ with eigenvalue $\lambda = \omega ^2$.
If the edge $e$ from $v$ to an adjacent vertex $w_i$ is identified with $[0,1]$, then 
\begin{equation} \label{eform}
y(x) = y(v)\cos (\omega x) + \frac{y(w_i) - y(v)\cos(\omega )}{\sin(\omega )} \sin (\omega x), \quad \sin(\omega ) \not= 0. 
\end{equation}
That is, on the edge $e$, $y$ may be recovered from its values at $0,1$ except when $ \sin(\omega ) = 0$.
The derivative condition \eqref{derive} gives
\[ y(v) \cos (\omega ) = \frac{1}{{\rm deg}(v)}\sum_j y(w_j) .\]
The vertex values of $y$ are thus an eigenvector for $T^{-1}A$ with eigenvalue  $\cos (\omega ) $,
so $\nu = 1 - \cos (\sqrt{\lambda } )$ is an eigenvalue of $\Delta $. 
Running the argument in reverse finishes the proof.
\end{proof}

For computations it is also helpful to note that \eqref{eform} constructs eigenfunctions of $\dop $ from
eigenvectors of $\Delta $.  The algorithm proceeds by computing the eigenvalues of the real symmetric matrix $I - T^{-1/2}A_1T^{-1/2}$,
along with the (dot product) orthonormal eigenvectors.   Multiplying the eigenvectors by $T^{-1/2}$ produces 
$\H $ orthonormal eigenvectors of $I - T^{-1}A$.
 
\subsection{Eigenfunction inner products}

Orthogonality of eigenspaces with distinct eigenvalues is a consequence of the self-adjointness of $\Delta $ and $\dop $.
This fact still leaves open the more detailed question about inner products of the constructed eigenfunctions $\Psi _{m,k}$.  
Suppose $\phi _1, \phi _2$ are eigenvectors of $\Delta $ with the same eigenvalue $\nu $.
For each edge $e$ and $j = 1,2$, let $a_{e,j} = \phi _j(0)$, $b_{e,j }= \phi _j(1)$.
Similarly assume that the functions $y_j$, with $y_j(0) = \phi _j (0)$ and $y_j(1) = \phi _j (1)$, are eigenfunctions of $\dop $ 
with eigenvalue $\lambda $ where $\nu = 1 - \cos (\sqrt{\lambda } ) $, $\lambda  \notin \{ n^2 \pi ^2 \ | \ n = 0,1,2,\dots \} $.
On each edge $e$ the eigenfunctions $y_j$ have the form
\[y_{e,j}(x,\lambda ) = A_{e,j}\cos (\omega x) + B_{e,j}\sin (\omega x), \quad \omega = \sqrt{\lambda } \ge 0. \]

\begin{thm} \label{eqdot}
Suppose $\lambda \notin \{ n^2\pi ^2 \}$.  Then
\begin{equation} \label{dot1}
\sum_{v \in \vset}{\rm deg} (v) \phi_1(v)\overline{\phi _2(v)} = \sum_{e \in \eset} (a_{e,1}\overline{a_{e,2}} + b_{e,1}\overline{b_{e,2}}) \end{equation}
\[= \sum_{e \in \eset} (A_{e,1}\overline{A_{e,2}} + B_{e,1}\overline{B_{e,2}}) .\]
\end{thm}

\begin{proof}
The first equality is simply the observation that a vertex $v$ is an endpoint of ${\rm deg }(v) $ edges.

As noted in \propref{combspec}, if $\lambda \notin \{ n^2\pi ^2 \}$ then 
\[A_{e,j} = a_{e,j}, \quad B_{e,j} = \frac{b_{e,j} - a_{e,j}\cos(\omega )}{\sin(\omega )} .\]
There is another eigenfunction $y_3$ of $\dop $ with eigenvalue $[2\pi - \omega ] ^2$ having the form
\[y_3(x,[2\pi - \omega ]^2) = a_{e,2}\cos ([2\pi - \omega ]x) + \frac{b_{e,2} - a_{e,2}\cos(2\pi - \omega  )}{\sin(2\pi - \omega )} \sin ([2\pi -\omega ] x) \]
\[= a_{e,2}\cos ([2\pi - \omega ]x) - \frac{b_{e,2} - a_{e,2}\cos(\omega  )}{\sin(\omega )} \sin ([2\pi -\omega ] x). \]

As eigenfunctions of $\dop $ with distinct eigenvalues, the functions $y_1,y_3$ are orthogonal.  Thus
\begin{equation} \label{inprodid}
0 = \int_{\graph } y_1\overline{y_3} =  \sum _{e \in \eset} \int_0^1 \Bigl [ a_{e,1}\cos (\omega x) + \frac{b_{e,1} - a_{e,1}\cos(\omega )}{\sin(\omega )} \sin (\omega x) \Bigr ] 
\end{equation}
\[ \times  \Bigl [ \overline{a_{e,2}\cos ([2\pi - \omega ]x) - \frac{b_{e,2} - a_{e,2}\cos(\omega  )}{\sin(\omega )} \sin ([2\pi -\omega ] x)} \Bigr ] \ dx .\]
Using elementary integrals and trigonometric identities, \eqref{inprodid} becomes
\[0 = \sum_{e \in \eset} a_{e,1}\overline{a_{e,2}}[ \sin (\omega )\cos (\omega ) - 2 \sin (\omega )\cos (\omega ) - \cos ^2(\omega )\frac{\cos (\omega )}{\sin (\omega )} ]\]
\[ + (a_{e,1}\overline{b_{e,2}}  + \overline{a_{e,2}}b_{e,1})[\sin (\omega ) + \frac{\cos ^2(\omega )}{\sin(\omega )}]  - b_{e,1}\overline{b_{e,2}}\frac{\cos (\omega )}{\sin(\omega )}  ,\]
and further simplification leads to
\begin{equation} \label{inprod3}
0 =  \sum_{e \in \eset} - a_{e,1}\overline{a_{e,2}}\cos (\omega) + (a_{e,1}\overline{b_{e,2}} + \overline{a_{e,2}}b_{e,1})  - b_{e,1}\overline{b_{e,2}}\cos(\omega ). 
\end{equation}

Rewrite \eqref{inprod3} first as
\[ 0 = \sum _{e \in \eset} (a_{e,1}\overline{a_{e,2}} + b_{e,1}\overline{b_{e,2}})\cos ^2(\omega ) -   (a_{e,1}\overline{b_{e,2}} + \overline{a_{e,2}}b_{e,1} )(\cos(\omega ))\]
Then using $\cos ^2(\omega ) = 1 - \sin ^2(\omega )$,
\[ \sin^2(\omega ) \sum _{e \in \eset} (a_{e,1}\overline{a_{e,2}} + b_{e,1}\overline{b_{e,2}}) \]
\[ = \sin ^2(\omega )\sum _{e \in \eset} a_{e,1}\overline{a_{e,2}} + \sum _{e \in \eset} a_{e,1}\overline{a_{e,2}} \cos ^2(\omega ) +  b_{e,1}\overline{b_{e,2}} - (a_{e,1}\overline{b_{e,2}} + \overline{a_{e,2}}b_{e,1} )(\cos(\omega ))\]
\[=  \sin ^2(\omega )\sum _{e \in \eset} a_{e,1}\overline{a_{e,2}} + \sum _{e \in \eset}\bigl [b_{e,1} - a_{e,1}\cos (\omega )\bigr] \bigl [\overline{b_{e,2} - a_{e,2}\cos (\omega )}\bigr] .\]
This establishes \eqref{dot1}.
\end{proof}

The next inner product result avoids the restriction $\lambda \notin \{ n^2\pi \}$.

\begin{thm} \label{inprod2}
For $j  = 1,2$, assume the functions $y_j$ are eigenfunctions of $\dop $ with the common eigenvalue $\lambda > 0$.
If  $y_j(x,\lambda ) = A_{e,j}\cos (\omega x) + B_{e,j}\sin (\omega x)$ on the edges $e$,  then
\begin{equation} \label{case2}
\int_{\graph } y_1\overline{y_2} = \sum_{e \in \eset} \frac{1}{2} [ A_{e,1}\overline{A_{e,2}} + B_{e,1}\overline{B_{e,2}}]
\end{equation}

\end{thm}

\begin{proof}

Simple calculations lead to 
\[\int_{\graph} y_1\overline{y_2}  = \sum_{e \in \eset} \frac{1}{2} [ A_{e,1}\overline{A_{e,2}} + B_{e,1}\overline{B_{e,2}}]  \]
\[+ \frac{\sin (\omega )}{2\omega} [ \cos (\omega ) A_{e,1}\overline{A_{e,2}} + \sin (\omega ) (A_{e,1}\overline{B_{e,2}} + \overline{A_{e,2}}B_{e,1}) - \cos (\omega ) B_{e,1}\overline{B_{e,2}} ] \]

Here the argument splits into two cases.  If $\lambda \in \{ n^2 \pi ^2 \}$, then $\sin (\omega ) = 0$ and \eqref{case2} is established.
For the cases $\lambda \notin \{ n^2 \pi ^2 \}$, proceed as follows.  Inserting the endpoint values $a_{e,j},b_{e,j}$ gives
\[ \cos (\omega ) A_{e,1}\overline{A_{e,2}} + \sin (\omega ) (A_{e,1}\overline{B_{e,2}} + \overline{A_{e,2}}B_{e,1}) - \cos (\omega ) B_{e,1}\overline{B_{e,2}} \]
\[ = - \cos (\omega )  a_{e,1}\overline{a_{e,2}} + a_{e,1}\overline{b_{e,2}}   + \overline{a_{e,2}}b_{e,1} \]
 \[ - \frac{ \cos (\omega )}{\sin ^2(\omega )}[  b_{e,1}\overline{b_{e,2}}  - a_{e,1}\overline{b_{e,2}}\cos (\omega ) - \overline{a_{e,2}}b_{e,1}\cos (\omega ) + a_{e,1}\overline{a_{e,2}}\cos ^2(\omega )]\]
 Using \eqref{inprod3}
 \[  \cos (\omega ) A_{e,1}\overline{A_{e,2}} + \sin (\omega ) (A_{e,1}\overline{B_{e,2}} + \overline{A_{e,2}}B_{e,1}) - \cos (\omega ) B_{e,1}\overline{B_{e,2}} \]
 \[ =  \cos (\omega )  b_{e,1}\overline{b_{e,2}}  - \frac{ \cos (\omega )}{\sin ^2(\omega )}[  b_{e,1}\overline{b_{e,2}}  - a_{e,1}\overline{a_{e,2}}\cos ^2(\omega ) - \overline{b_{e,2}}b_{e,1}\cos ^2(\omega ) + a_{e,1}\overline{a_{e,2}}cos ^2(\omega )]  \]
 \[ = 0.\]
 This establishes \eqref{case2}.
\end{proof}

\begin{cor} \label{freqinc}
Suppose $z_1$ and $z_2$ are eigenfunctions of $\dop $ with common eigenvalue $[\omega +2m\pi ]^2$, obtained from $y_1$ and $y_2$ with common eigenvalue $\omega ^2$ by increasing the frequency.
If $y_1$ and $y_2$ are orthonormal, so are $z_1$ and $z_2$. 
\end{cor}

\subsection{Sampling}   

Since the QGFFT uses values of functions at sample points, a couple of additional issues arise. 
The insertion of sample points to the edges of $\graph $ may be viewed as an equivalent quantum graph with
each edge replaced by a path of $N$ edges of length $1/N$.  The Hilbert space of complex functions defined on the sample points
with an inner product  $\langle \cdot , \cdot \rangle _N$ similar to \eqref{inprod} is denoted $\H _N$.  The associated discrete Laplacian is $\Delta _N$. 

The truncated Fourier series used for the QGFFT should be capable of representing functions in $\H _N$.
Let $E_p(N^2\pi ^2)$ denote the subspace of the eigenspace $E(N^2\pi ^2)$ spanned by eigenfunctions 
of $\dop $ having the form $\pm \cos(N\pi x)$ on each edge. This subspace has dimension at most $1$, and its nonzero elements do not vanish at the vertices.
Define the truncated Fourier space, 
\begin{equation} \label{DefSN}
S_N = span \Bigl \{ \{ E(\lambda ), 0 \le \lambda < N^2\pi ^2 \} \cup \{ E_p(N^2\pi ^2)\} \Bigr \}.
\end{equation}
For continuous functions $f:\graph _1 \to \complex$, define the restriction map $R_N$ taking $f$ to its values on edge samples.
  
The following result is proven using a rescaled version of \propref{combspec}; the proof can be found in \cite[Prop 3.2]{CarlsonFFT}.
The condition $0 < \omega _{m,k} \le N\pi $ here amounts to $m = 0,\dots , N/2 -1$ in the Fourier series of \eqref{betadef}.

\begin{prop} \label{biject}
The restriction map $R_N:S_N \to \H _N$ is a bijection. For $0 \le \lambda < N^2 \pi ^2$
this map takes distinct  eigenspaces $E(\lambda )$ 
onto distinct eigenspaces $E_N(N^2(1 - \cos (\sqrt{\lambda }/N))$ of $\Delta _N$, and $R_N$
takes $E_p (N^2\pi ^2) $ onto $E_N(2N^2)$.
\end{prop}

Normally the trapezoidal rule, which is used for $\langle \cdot , \cdot \rangle _N$, gives accurate, but not exact values for integrals.
However, except for a slight modification for $E_p (N^2\pi ^2) $, and mainly by using \thmref{eqdot}, 
one can show that $\langle \Psi _{m,j}, \Psi _{m,k} \rangle _N = \int_{\graph} \Psi _{m,j}\overline{\Psi _{m,k}} $.

This desired orthonormality was tested computationally  in several cases:
the complete graph on $4$ vertices, a figure $8$, and the edges of a $3d$ cube.
Diagonal and off-diagonal inner products 
\[ \langle \Psi _{j,m} ,\Psi _{k,m}\rangle = \sum_{e \in \eset} \int_0^1 \Psi _{j,m}(x) \overline{\Psi _{k,m}(x)} \ dx \]
were computed using the trapezoidal rule on each edge, with $N = 64$. 
In all cases the diagonal inner products satisfied $ |\langle \Psi _{k,m} ,\Psi _{k,m}\rangle - 1 | < 10^{-14}$.
The accuracy for the off-diagonal terms also satisfied  $ |\langle \Psi _{j,m} ,\Psi _{k,m}\rangle | < 10^{-14}$.

\section{Spectral and pseudospectral methods}

Below are several applications of both spectral and a pseudospectral methods for solving the Schr\"odinger equation 
\begin{equation} \label{Schrod1}
\frac{\partial \psi (t,x) }{\partial t} = i \dop \psi (t,x), \quad \psi (0,x) = \psi_0(x) 
\end{equation}
on networks.  Networks used for sample computations are shown in Figure \ref{fig:Loop}.  

The spectral method is most directly applicable when $ \dop = - \partial ^2/\partial x^2$ and the edges of $\graph $ have length $1$.
When the edge lengths are integers, the original edges can be replaced by paths, leaving edges with length $1$.
With respect to the natural identification, this vertex insertion leaves the operator  $- \partial ^2/\partial x^2$ unchanged. 
The QGFFT provides a computed version of a truncated eigenfunction expansion.  Explicit formulas are thus available to compute  
$\psi (t,x)$ at arbitrary times $t$ with extraordinary accuracy and efficiency, without the stability concerns plaguing other methods.
The examples below include a case with a nonconstant potential, $ \dop = - \frac{\partial ^2}{\partial x^2} + p(x)$. 
For such problems the Strang splitting method \cite{Strang} can be used to alternately advance spectral solutions of 
$\partial \psi /\partial t = -i \frac{\partial ^2\psi }{\partial x^2}$, and then (explicit) solutions of $\partial \psi /\partial t = i p(x)$.

\begin{figure}[h!tbp] \centering
\hfill
 \begin{subfigure}[b]{0.4\textwidth}
 \includegraphics[width= 2in, height = 1.65 in]{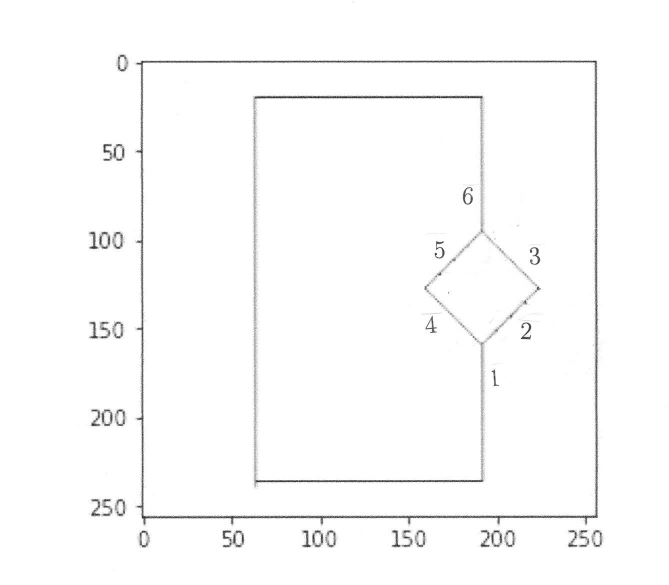}
 \caption{Loop with box obstacle}
 \end{subfigure}
 \centering
\hfill
 \begin{subfigure}[b]{0.4\textwidth}
 \includegraphics[width= 2in, height = 1.65 in]{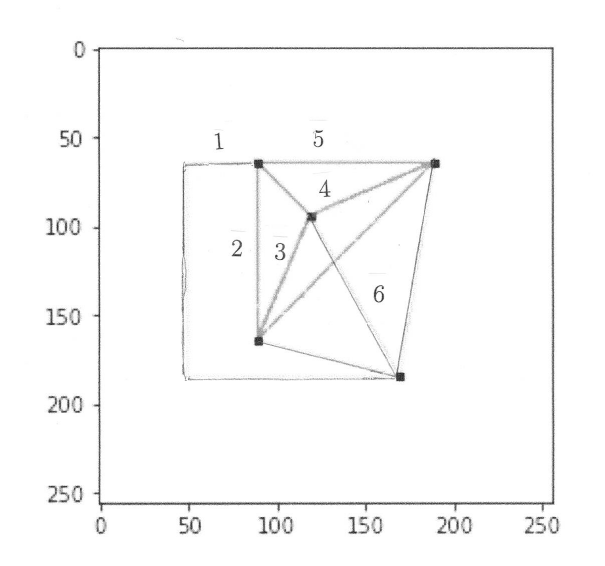}
 \caption{$K_4$ inside $K_5$}
 \end{subfigure}
 \caption{Sample networks}
 \label{fig:Loop}
\end{figure}

\subsection{Pseudospectral method}

If the edges of $\graph $ have lengths which are not integer multiples of a common value, a different vertex insertion strategy is used 
in preparation for the pseudospectral method.  This step is described in the following lemma.  

\begin{lem} \label{pscuts}
It is possible to insert a finite collection of vertices and rescale the edges of $\graph $ so the longest edge has length $1$ and the shortest edge length is at least $2/3$.
\end{lem}

\begin{proof}
If $l_{min}$ is the length of the shortest edge, the length of any edge $e$ can be expressed as $l_e = m_el_{min} + n_e$, with $m_e$ an integer and $0 \le n_e  < l_{min}$.
Insert vertices to replace $e$ by a path with $m$ edges of length $\mu _e = l_{min} + n_e/m_e$.
If $n_e/m_e \le l_{min}/2$ for all edges $e$, the vertex insertion is finished.  

In case there are edges with $l_{min} < l_e < 2l_{min}$ ($m_e = 1$), cut the edges with lengths $l_{min}$ and $\mu _e$ in half.  
Replace edges with lengths $l_{min} < l_e < 3l_{min}/2$ by two edges with lengths $l_e/2$.
Replace edges with lengths $3l_{min}/2 \le l_e < l_{min}$ by three edges with lengths $l_e/3$, noting that
$l_{min}/2 \le l_e/3 \le l_{min}/3 < 3l_{min}/4$.

Now the longest edge has length at most $3/2$ times the length of the shortest edge.
Rescaling the edges so the largest edge has length $1$ completes the proof.
\end{proof}

After taking advantage of \lemref{pscuts} the edges $e$ of $\graph $ will have lengths $l_e$ with $2/3 \le l_e \le 1$.
For each edge $e$ with local coordinate $x$, introduce new variables $\xi = \phi_e (x)$ with
\begin{equation} \label{COV}
\phi _e(x) = \frac{x}{l_e} - \frac{1-l_e}{2\pi }\sin (2\pi x/l_e) .
\end{equation}
Since $\phi_e '(x) > 0$, the function $\phi _e:[0,l_e] \to [0,1]$ is strictly increasing, with derivative $1$ at both ends of the interval, so functions 
which satisfy the Kirchhoff boundary conditions \eqref{derive} are preserved after the change of variables. 

Letting $\sigma_e (\xi) = \phi_e '(\phi_e ^{-1}(\xi))$,  the chain rule gives
\begin{equation} \label{newLap} 
\frac{d}{dx} = \phi_e '(x) \frac{d}{d\xi} = \sigma (\xi)\frac{d}{d\xi}, \quad  
\frac{d^2}{dx^2} = \sigma_e (\xi)  \frac{d}{d\xi} \sigma_e (\xi) \frac{d}{d\xi}.
\end{equation}
A conjugation leads to 
\[\sigma_e ^{1/2} \sigma_e  \frac{d}{d\xi} \sigma_e  \frac{d}{d\xi} \sigma_e ^{-1/2}
=  \sigma_e ^{2} \frac{d^2}{d\xi ^2} +  \frac{1}{4} (\sigma_e ')^2 - \frac{1}{2} \sigma_e \sigma_e '' \]
The simple form of \eqref{COV} leads to 
\[\sigma_e '(\xi ) = \phi_e ''(\phi_e ^{-1}(\xi ))\frac{1}{\sigma_e (\xi )}, \quad \phi_e ''(x) =  \frac{2\pi (1-l_e)}{l_e^2}\sin (\frac{2\pi x}{l_e}),\]
and 
\[ \sigma_e ''(\xi )  =  \phi_e ^{(3)}(\phi_e ^{-1}(\xi ))\frac{1}{\sigma_e ^2} -  \phi_e ''(\phi_e ^{-1}(\xi ))\frac{\sigma_e '}{\sigma_e ^2} ,\quad
\phi_e ^{(3)}(x) =  \frac{(2\pi )^2(1-l_e)}{l_e^3}\cos (\frac{2\pi x}{l_e}).\]

The goal now is to compute numerical approximations of solutions to the time dependent Schr\"odinger equation
\[ \frac{\partial \psi (t,x) }{\partial t} = i \dop \psi (t,x), \quad \psi (0,x) = \psi_0(x) .\]
With $\Psi (t,\xi ) = \sigma ^{1/2}(\xi )\psi (\phi ^{-1}(\xi ))$, the equivalent version 
\begin{equation} \label{realform}
 -i \frac{\partial \Psi}{\partial t}=  - \sigma_e ^{2} \frac{d^2 \Psi }{d\xi ^2} -  \frac{1}{4} (\sigma_e ')^2\Psi  + \frac{1}{2} \sigma_e \sigma_e ''\Psi , \quad 
\Psi (0,\xi ) = \sigma ^{1/2}(\xi )\psi _0(\phi ^{-1}(\xi )) 
\end{equation}
is treated on the graph $\graph _{\xi}$ with all edges of length $1$.  Edges have samples $\xi _0,\xi _1 ,\dots ,\xi _N$ with equal spacing $\Delta \xi = 1/N$.
The values $\phi ^{-1}(\xi _j)$ are computed by interpolation.
Unless otherwise indicated, the value $N = 32$ was used.   

A standard integration scheme is used to compute approximations $U^{n}_j \sim \Psi (t_n,\xi _j)$ at times $t_n$ with equal spacing $h$;
$h$ is taken small enough to avoid numerical instability.
The scheme, which takes advantage of the form of the Schr\"odinger equation (but is not suitable for the heat equation \cite[p. 35]{Morton}) is 
\begin{equation} \label{Askar} 
U^{n+1}_j = U^{n-1}_j + \mu (i\dop U)_j, \quad \mu = \frac{2h}{(\Delta \xi )^2}.
\end{equation}
This method was proposed in \cite{Harmuth}, then later in \cite{Askar}, and was also used with a pseudospectral method in \cite{Kosloff}.  
Extensions were recently considered in \cite{vanDijk}. 
As noted in \cite{Rubin}, the extravagant stability claim in \cite{Askar} is erroneous. 

The pseudospectral method uses the QGFFT and its inverse to calculate $d^2 \Psi /d\xi ^2 $ at the spatial samples of the graph.
The updates are then computed in the spatial domain.  For comparison, some calculations are also carried out with a difference operator, which 
approximates $d^2 \Psi /d\xi ^2 $ at a sample point $v$,
\begin{equation} \label{diff}
\frac{d^2 \Psi }{d\xi ^2} \sim \frac{2}{\deg(v)(\Delta \xi ) ^2}[ - \deg(v) f_e(v) + \sum_{w} f_e(w ) ],
\end{equation}
the sum taken over sample points $w$ adjacent to $v$.  Notice that this is just the usual second difference operator 
at samples interior to the edges.

After the computation of $\Psi (t_n,\xi _j)$ is completed, the solution is mapped back to the original space variable.
The number of spatial samples for the edge $e$ is $int(N*l_e)$ where $l_e$ is the edge length after vertex insertions.
Interpolation is used to determine $\psi (t_n,x_j)$. Recall that length rescaling is modest since $2/3 \le l_e \le 1$.    

\subsection{Computations}

\subsubsection{Accuracy}

The time dependent Schr\"odinger equation \eqref{Schrod1} with $\dop = -\partial ^2/\partial x^2$ 
is first considered on the network of \ref{fig:Loop}, subfigure (a), a loop with box obstacle.  
To provide a comparison of the spectral method with the pseudospectral and difference methods
the network has all $11$ edges with length $1$.
In each case the initial function $\psi _0(x)$ has the value $\pm \sin(m\pi x)$ on the edges of the box, while $\psi _0(x)$ is zero on the other edges.
The $\pm $ signs are chosen to satisfy the vertex conditions.  The algorithms were tested with values $m = 1,3,5,7$.
With these values $\psi _0(x)$ is an eigenfunction of $\dop$ with eigenvalues $\lambda _m= m^2\pi^2$.  The exact solutions are then
$\psi (t,x) = \exp(i \lambda _m t )\psi _0(x)$ on the edges of the box, and zero elsewhere. 

Three methods for solving \eqref{Schrod1} are compared in Figure \ref{fig:Accur}, subfigures $(a)-(c)$.  
In each case solutions $\phi (t_k,x_n)$ are computed from $t = 0$ to $t = 1/(16\pi)$ using an increment $h$ resulting in $200$ time samples.  
Figure \ref{fig:Accur} shows values of a measure of relative error 
\begin{equation} \label{relerr}
\log_{10}\Bigl ( \Bigl[ \frac{ \int_{\graph} |\psi (t_k,x) - \phi (t_k,x)|^2 \ dx}{\int_{\graph} |\psi(t,x)|^2 \ dx} \Bigr ]^{1/2} \Bigr )
\end{equation}
on a logarithmic scale as a function of time.  The integrals are calculated using the trapezoidal rule on each edge.

\begin{figure}[h!tbp] 
\centering
 \begin{subfigure}[b]{0.4\textwidth}
 \includegraphics[width= 1.75in, height = 1.2 in]{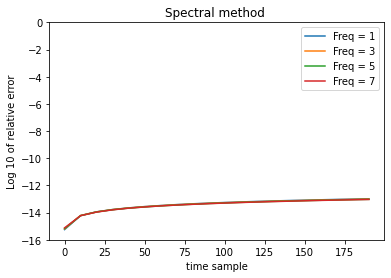}
 \caption{Spectral method accuracy }
\end{subfigure}
\centering
\hfill
\begin{subfigure}[b]{0.4\textwidth}
\includegraphics[width= 1.75in, height = 1.2 in]{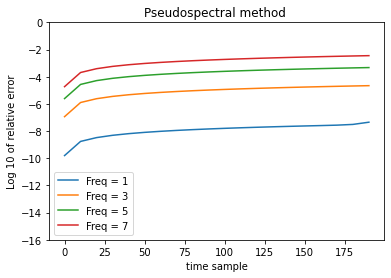}
\caption{Pseudospectral accuracy}
\end{subfigure}
\centering
\hfill
 \begin{subfigure}[b]{0.4\textwidth}
 \includegraphics[width= 1.75in, height = 1.2 in]{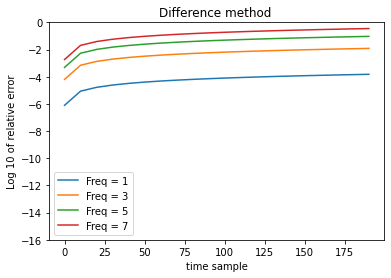}
 \caption{Difference accuracy }
\end{subfigure}
\centering
\hfill
\begin{subfigure}[b]{0.4\textwidth}
\includegraphics[width= 1.75in, height = 1.2 in]{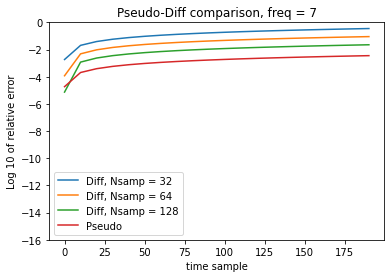}
\caption{Difference vs. Pseudospectral}
\end{subfigure}
\caption{Method accuracy}
\label{fig:Accur}
\end{figure}

The spectral method of subfigure $(a)$ advances solutions by time increment $h$ by multiplying Fourier components with eigenvalue $\lambda $ by
$\exp(i\lambda h)$, mimicking the exact solution using Fourier series.  Computational errors are smaller than $10^{-13}$;  these errors are largely independent of 
the frequency of the initial data.  The pseudospectral method advances solutions using \eqref{Askar}
after initially computing $\phi (t_1,x_n)$ from $\phi (0,x_n)$ using a two step Runge-Kutta method. 
The third scheme uses \eqref{Askar} together with the difference approximation \eqref{diff}.  Both the pseudospectral method and the difference operator method
show an expected frequency sensitivity.

In these experiments the difference scheme, with $N = 32$, runs about thirty times faster than
either the spectral or pseudospectral method.  Subfigure (d) of Figure \ref{fig:Accur} shows the results of trying to obtain pseudospectral accuracy with the difference method by increasing the number of samples on each edge.  The frequency case $m = 7$ is tested. To maintain stability, the ratio $h/ (x_{n+1}-x_n)^2$ is kept fixed.  The pseudospectral method is considerably more accurate even after the difference method has lost its speed advantage.

\begin{figure}[h!tbp] 
\centering
\begin{subfigure}[b]{0.4\textwidth}
\includegraphics[width= 1.75in, height = 1.2in]{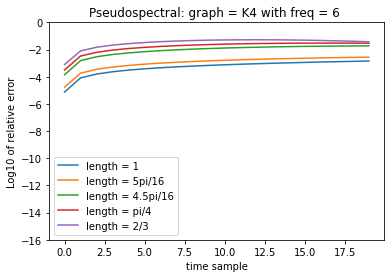}
\caption{Length dependence}
\end{subfigure}
\centering
\hfill
\begin{subfigure}[b]{0.4\textwidth}
\includegraphics[width= 1.75in, height = 1.2 in]{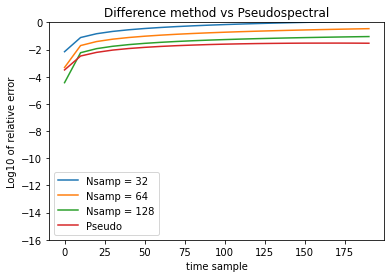}
\caption{Difference vs. Pseudospectral}
\end{subfigure}
\caption{Pseudospectral method}
\label{fig:K4}
\end{figure}

Figure \ref{fig:K4} subfigure (a) presents an accuracy test of the pseudospectral method for various edge lengths.
In this case \eqref{Schrod1} is solved on a complete graph with $4$ vertices, $K_4$.  
Three of the six edges have length $1$, while the remaining three have the same length $l$ with values $l = 2/3, \pi/4, 4.5\pi/16, 5\pi/16, 1$.
The initial function $\psi _0(x)$ has the value $\pm \sin(12\pi x/l)$ on the edges of the graph with length $l $, while $\psi _0(x)$ is zero on the other edges.
With these values $\psi _0(x)$ is an eigenfunction of $\dop = -\partial ^2/\partial x^2$ with eigenvalues $\lambda = (12 \pi /l)^2$.  
The relative error measure of \eqref{relerr} is used again.  Subfigure (b) compares the pseudospectral and difference methods with $l = \pi/4$.
After conversion to an equilateral graph, the pseudospectral method uses $N = 32$, while the difference operator is tested with
$N = 32,64,128$.  As above, the ratio $h/ (x_{n+1}-x_n)^2$ is kept fixed, and the pseudospectral method is more accurate.  

\begin{figure}[h!tbp] 
\centering
\begin{subfigure}[b]{0.3\textwidth}
\includegraphics[width= 1.5in, height = 1 in]{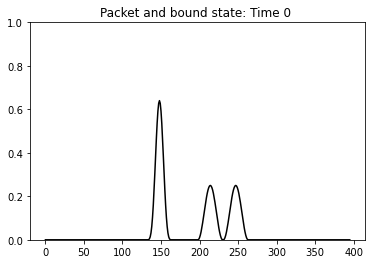}
\caption{Wave approaches }
\end{subfigure}
\centering
\hfill
\begin{subfigure}[b]{0.3\textwidth}
\includegraphics[width= 1.5in, height = 1 in]{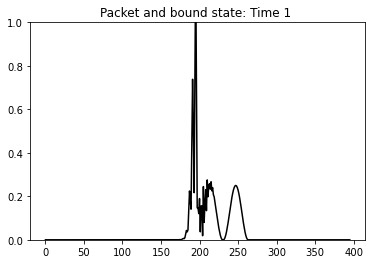}
\caption{Interaction begins}
\end{subfigure}
\centering
\hfill
 \begin{subfigure}[b]{0.3\textwidth}
 \includegraphics[width= 1.5in, height = 1 in]{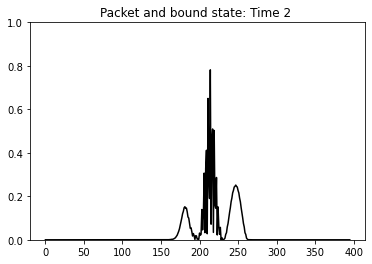}
 \caption{Strong interaction }
\end{subfigure}
\centering
\hfill
\begin{subfigure}[b]{0.3\textwidth}
\includegraphics[width= 1.5in, height = 1 in]{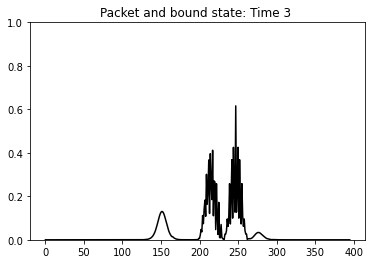}
\caption{Reflected packet separates}
\end{subfigure}
\centering
\hfill
 \begin{subfigure}[b]{0.3\textwidth}
 \includegraphics[width= 1.5in, height = 1 in]{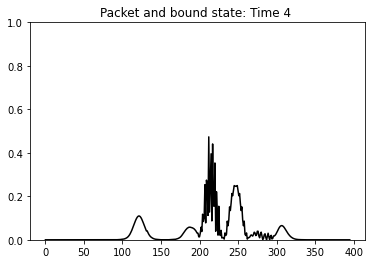}
 \caption{Complex reflection and transmission}
\end{subfigure}
\centering
\hfill
 \begin{subfigure}[b]{0.3\textwidth}
 \includegraphics[width= 1.5in, height = 1 in]{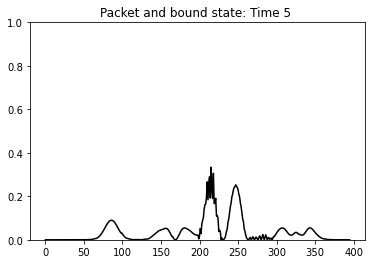}
 \caption{More reflection and transmission}
\end{subfigure}
\caption{Interaction of wave packet with bound state.}
\label{fig:Pack}
\end{figure}

\subsubsection{Three examples}

Figures  \ref{fig:Pack},  \ref{fig:Potential}, and \ref{fig:PackEQ} below illustrate applications of the QGFFT to the Schr\"odinger equation
on quantum graphs.  In each case the images show $|\psi (t,x)|^2$ for solutions of \eqref{Schrod1}; the sequence is 
sampled from a longer set describing a temporal evolution.  Figures \ref{fig:Pack} and \ref{fig:PackEQ} highlight a significant strength of the Fourier methods: the modeling of high frequency interference 
phenomena.

Figure \ref{fig:Pack} illustrates the use of the spectral method to compute the interaction of a wave packet with a graph supporting a geometric bound state.  
Such bound states were recently considered theoretically in \cite{Harrell}.
The total network, shown in Figure \ref{fig:Loop} subfigure (b), is a complete graph on $5$ vertices, $K_5$, with the $K_4$ subgraph viewed as an obstacle.
The edges of the $K_4$ have length $1$; the leftmost path joining edges $1$ and $6$ is longer.
   
The bound state $ \psi _0(x)$ is an eigenfunction of $\dop = - \partial ^2/\partial x ^2$, with 
values $\pm \sin(\pi x)$ on a cycle of four edges of the $K_4$ subgraph, labelled $2-5$ while $\psi _0(x)$ is zero on the other edges.
The values of $|\psi (t,x)|^2$ are shown along a path in $K_5$ which enters the $K_4$ subgraph along edge $1$,
continues along edges $2,3$, and then exits through edge $6$.
In addition to the superposition of the wave packet and the bound state on edges $2,3$, the subfigures show a series of packets generated 
by the interaction of the original wave packet with the $K_4$ subgraph.  
The multifeatured solution includes regions of high frequency interference in addition to reflected and transmitted packets.

Figure \ref{fig:Potential} uses the Figure \ref{fig:Loop} subfigure (a) graph to illustrate the evolution of  
bound state initial data (as above) when $\dop = - \partial ^2/\partial x ^2 + p(x)$.
The Strang splitting method is used, alternately advancing the solution with the spectral method for the zero potential,
and the exact solution of $ \partial \psi /\partial t = ip(x)\psi  $.    
The potential is zero except for the single edge $2$ where $p(x) = -25 (1 - \cos (2\pi x))$.  The initial bound state has support on the subgraph of edges $2-5$.
The values of $|\psi (t,x)|^2$ are shown along a path containing edges $1,2,3,6$.  Solution values on edges $4,5$ are displayed on the far right.
Since the initial data is no longer an eigenfunction of $\dop $ it 'leaks' away from the initial support.

\begin{figure}[h!tbp] 
\centering
\begin{subfigure}[b]{0.3\textwidth}
\includegraphics[width= 1.5in, height = 1 in]{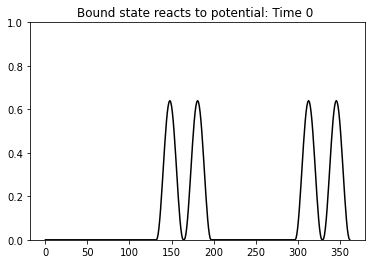}
\caption{Initial bound state }
\end{subfigure}
\centering
\hfill
\begin{subfigure}[b]{0.3\textwidth}
\includegraphics[width= 1.5in, height = 1in]{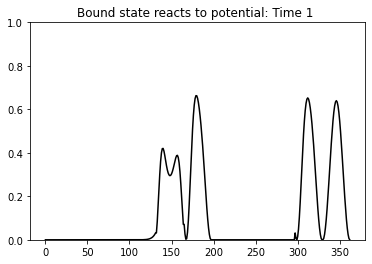}
\caption{Time 1}
\end{subfigure}
\centering
\hfill
 \begin{subfigure}[b]{0.3\textwidth}
 \includegraphics[width= 1.5in, height = 1 in]{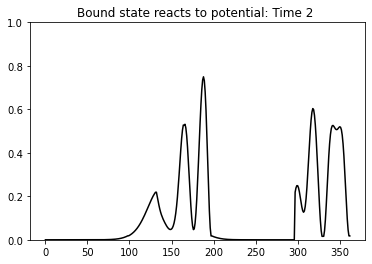}
 \caption{Time 2}
\end{subfigure}
\centering
\hfill
\begin{subfigure}[b]{0.3\textwidth}
\includegraphics[width= 1.5in, height = 1in]{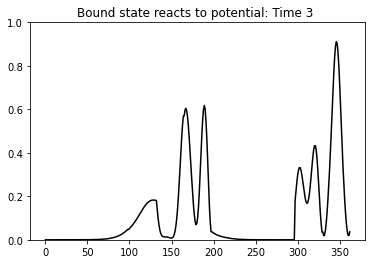}
\caption{Time 3}
\end{subfigure}
\centering
\hfill
 \begin{subfigure}[b]{0.3\textwidth}
 \includegraphics[width= 1.5in, height = 1 in]{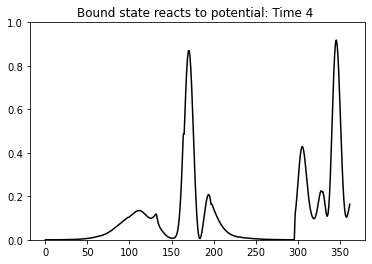}
 \caption{Time 4}
\end{subfigure}
\centering
\hfill
 \begin{subfigure}[b]{0.3\textwidth}
 \includegraphics[width= 1.5in, height = 1 in]{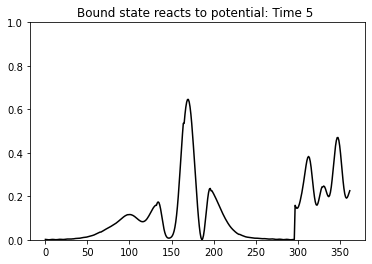}
 \caption{Time 5}
\end{subfigure}
\caption{Initial bound state with potential.}
\label{fig:Potential}
\end{figure}

The final computations, shown in  Figure \ref{fig:PackEQ}, employ the pseudospectral method.  The network is shown in Figure \ref{fig:Loop}, subfigure (a).
In this case edge $4$ has length $\pi /4$, edge $5$ has length $\sqrt{2}/2$, and edges $2$ and $3$ have length $1$.
The figures show a wave packet hitting the box from edge $1$.  The horizontal lines under the axis mark the positions of
edges $2,3$, and then at the right edges $4,5$.  

In subfigure (b) the wave packet is seen encountering the first vertex of the box.
There is then a strong interaction between the incoming and reflected packets.  
In subfigure (d) the $4,5$ packet, which travels over a shorter path, encounters the last vertex of the box before the $2,3$ packet does.
There is a strong interaction within edge $3$.
Subfigures (e) and (f) show reflected packets from the initial encounter, reflected packets in the box, and a packet which has passed through the box.

\begin{figure}[h!tbp] 
\centering
\begin{subfigure}[b]{0.3\textwidth}
\includegraphics[width= 1.5in, height = 1 in]{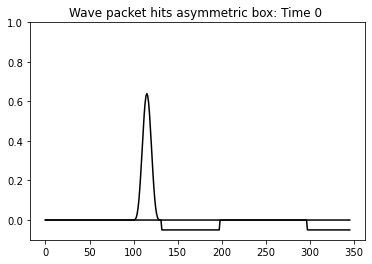}
\caption{Wave approaches }
\end{subfigure}
\centering
\hfill
\begin{subfigure}[b]{0.3\textwidth}
\includegraphics[width= 1.5in, height = 1in]{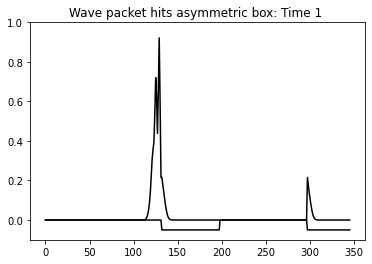}
\caption{Hitting leading edge}
\end{subfigure}
\centering
\hfill
 \begin{subfigure}[b]{0.3\textwidth}
 \includegraphics[width= 1.5in, height = 1 in]{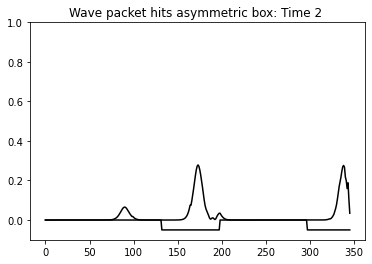}
 \caption{Front edge reflection}
\end{subfigure}
\centering
\hfill
\begin{subfigure}[b]{0.3\textwidth}
\includegraphics[width= 1.5in, height = 1in]{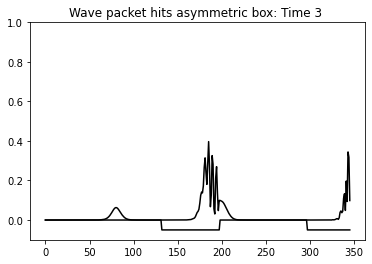}
\caption{Interaction of two waves}
\end{subfigure}
\centering
\hfill
 \begin{subfigure}[b]{0.3\textwidth}
 \includegraphics[width= 1.5in, height = 1 in]{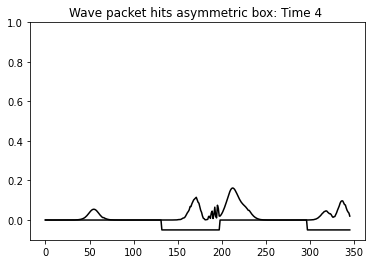}
 \caption{Reflections and transmissions}
\end{subfigure}
\centering
\hfill
 \begin{subfigure}[b]{0.3\textwidth}
 \includegraphics[width= 1.5in, height = 1 in]{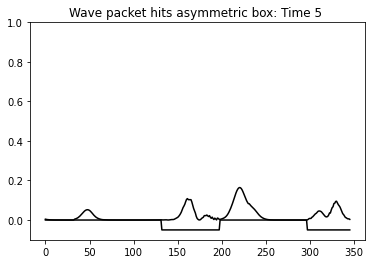}
 \caption{Later reflections and transmissions}
\end{subfigure}

\caption{Wave packet hits asymmetric box}
\label{fig:PackEQ}
\end{figure}

\bibliography{mybiblio}

\begin{thebibliography}{10}

\bibitem{Arioli}
M.~Arioli and M.~Benzi.
\newblock A finite element method for quantum graphs.
\newblock {\em IMA Journal of Numerical Analysis}, 38(3):1119--1163, 2018.

\bibitem{Askar}
A.~Askar and A.~Cakmak.
\newblock Explicit integration method for the time-dependent schrodinger
  equation for collision problems.
\newblock {\em J. Chem. Phys.}, 68(6):2794--2798, 1978.

\bibitem{BK}
G.~Berkolaiko and P.~Kuchment.
\newblock {\em An Introduction to Quantum Graphs}.
\newblock American Mathematical Society, Providence, 2013.

\bibitem{Briggs}
W.~Briggs and Van~Emden Henson.
\newblock {\em The DFT: an owners manual for the discrete Fourier transform}.
\newblock SIAM, Philadelphia, 1995.

\bibitem{Besse}
S.~LeCoz C.~Besse, R.~Duboscq.
\newblock Numerical simulations on nonlinear quantum graphs with the grafidi
  library.
\newblock {\em SMAI Journal of Computational Mathematics}, pages 1--47, 2022.

\bibitem{CarlsonFFT}
R.~Carlson.
\newblock Harmonic analysis for graph refinements and the continuous graph fft.
\newblock {\em Linear Algebra and Its Applications}, 430(11-12):2859--2876,
  2009.

\bibitem{CarlsonArXiv}
R.~Carlson.
\newblock A quantum graph fft with applications to partial differential
  equations on networks.
\newblock preprint version 1, arXiv: 2410.19969, 2024.

\bibitem{Cattaneo}
C.~Cattaneo.
\newblock The spectrum of the continuous laplacian on a graph.
\newblock {\em Monatsh. Math.}, 124(3):215--235, 1997.

\bibitem{Chung}
F.~Chung.
\newblock {\em Spectral Graph Theory}.
\newblock American Mathematical Society, Providence, 1997.

\bibitem{Forn}
B.~Fornberg.
\newblock {\em A Practical Guide to Pseudospectral Methods}.
\newblock Cambridge University Press, Cambridge, UK, 1998.

\bibitem{Harmuth}
H.F. Harmuth.
\newblock On the solutions of the schoedinger and klein-gordon equations by
  digital computers.
\newblock {\em Journal of Mathematics and Physics}, 36:269--278, 1957.

\bibitem{Harrell}
E.~Harrell and A.~Maltsev.
\newblock On topological bound states and secular equations for quantum-graph
  eigenvalues.
\newblock {\em Journal of Spectral Theory}, 14(2):619--639, 2024.

\bibitem{Hesthaven}
D.~Gottlieb J.~Hesthaven, S.~Gottlieb.
\newblock {\em Spectral Methods for Time-Dependent problems}.
\newblock Cambridge University Press, Cambridge, UK, 2007.

\bibitem{Kosloff}
D.~Kosloff and R.~Kosloff.
\newblock A fourier method solution for the time dependent schr\"odinger
  equation as a tool in molecular dynamics.
\newblock {\em Journal of Computational Physics}, 52:35--53, 1983.

\bibitem{Kuchment}
P.~Kuchment and O.~Post.
\newblock On the spectra of carbon nano-structures.
\newblock {\em Comm. Math. Phy.}, 275(3):805--826, 2007.

\bibitem{Kurasov}
P.~Kurasov.
\newblock {\em Spectral Geometry of Graphs}.
\newblock Birkhauser, Berlin, 2024.

\bibitem{Brio}
H.~Kravitz M.~Brio, J.~Caputo.
\newblock Spectral solutions of pdes on networks.
\newblock {\em Appl. Numer. Math.}, 172:99--117, 2022.

\bibitem{Morton}
K.~Morton and D.~Mayers.
\newblock {\em Numerical Solution of Partial Differential Equations}.
\newblock Cambridge University Press, Cambridge, UK, 1994.

\bibitem{Goodman}
J.~Marzuola R.~Goodman, G.~Conte.
\newblock Qglab: A matlab package for computations on quantum graphs.
\newblock {\em SIAM Journal Scientific Computing}, 47(2):B428--B453, 2025.

\bibitem{Rubin}
R.~Rubin.
\newblock Comment on explicit integration method for the time-dependent
  schrodinger equation.
\newblock {\em J. Chem. Phys.}, 70:4811, 1979.

\bibitem{Strang}
G.~Strang.
\newblock On the construction and comparison of difference schemes.
\newblock {\em SIAM Journal on Numerical Analysis}, pages 506--517, 2968.

\bibitem{vanDijk}
W.~van Dijk.
\newblock On numerical solutions of the time-dependent schr\"odinger equation.
\newblock {\em American Journal of Physics}, 91(10):826--839, 2023.

\bibitem{vonBelow}
J.~von Below.
\newblock A characteristic equation associated to an eigenvalue problem on
  $c^2$ - networks.
\newblock {\em Linear Algebra and Its Applications}, 71(23):309--325, 1985.

\end{thebibliography}

\end{document}